\documentclass[11pt]{amsart}
\usepackage{geometry}                
\usepackage{graphicx}
\usepackage{amssymb}
\usepackage{epstopdf}
\usepackage[cmtip,arrow, all]{xy}
\title{Bredon cohomology and robot motion planning}
\date{}                                           
\author{Michael Farber}
\address{School of Mathematical Sciences \\
Queen Mary, University of London\\
London, E1 4NS\\
United Kingdom}
\email{m.farber@qmul.ac.uk}
\thanks{Michael Farber was partially supported by the EPSRC, by the IIAS and by the Marie Curie Actions, FP7, in the frame of the EURIAS Fellowship Programme}
\author{Mark Grant}
\address{Institute of Pure and Applied Mathematics \\
University of Aberdeen\\
Aberdeen AB24 3UE\\
United Kingdom}
\email{mark.grant@abdn.ac.uk}
\author{Gregory Lupton}
\address{Department of Mathematics\\
Cleveland State University\\
Cleveland OH 44115  \\
U.S.A.}
\email{g.lupton@csuohio.edu}
\thanks{
Gregory Lupton and John Oprea were partially supported by grants from the Simons Foundation (\#209575 and 
\#244393). }
\author{John Oprea}
\address{Department of Mathematics\\
Cleveland State University\\
Cleveland OH 44115  \\
U.S.A.}
\email{j.oprea@csuohio.edu}
\thanks{This research was supported through the programme \lq\lq Research in Pairs\rq\rq\,  by the Mathematisches Forschungsinstitut Oberwolfach in 2017}
\begin{document}
\maketitle

\newtheorem{theorem}{Theorem}[section]
\newtheorem{proposition}{Proposition}[subsection]
\newtheorem{lemma}[proposition]{Lemma}
\newtheorem{corollary}[proposition]{Corollary}
\newtheorem{definition}[proposition]{Definition}
\newtheorem{remark}[proposition]{Remark}
\newtheorem{example}[proposition]{Example}
\newcommand {\Ext}{{\rm {Ext}}}
\newcommand {\Hom}{{\rm {Hom}}}
\newcommand{\C}{{{\mathcal C}}}
\newcommand{\R}{{\mathbf R}}

\newcommand{\Q}{{\mathbf Q}}

\newcommand{\D}{{\mathcal D}}
\newcommand{\Z}{{\mathbb Z}}
\newcommand{\uz}{{\underline \Z}}
\newcommand{\uu}{{\mathfrak u}}
\newcommand{\vv}{{\mathfrak v}}
\newcommand{\od}{{\mathcal {O_D}}}
\newcommand{\odd}{{\mathcal {O_{D'}}}}
\newcommand{\ui}{{\underline I}}
\newcommand{\um}{{\underline M}}
\newcommand{\un}{{\underline N}}
\newcommand{\zpp}{{\Z[\pi\times\pi]}}
\newcommand{\zp}{{\Z[\pi]}}
\newcommand{\uc}{{\underline C}}
\newcommand{\tc}{{\sf {TC}}}
\newcommand{\secat}{{\sf {secat}}}
\newcommand{\uP}{{\underline P}}
\newcommand{\cat} {{\sf {cat}}}
\newcommand{\cd}{{\rm {cd}}}
\newcommand{\F}{{\mathcal {F}}}
\newcommand{\ualpha}{{\underline \alpha}}

\begin{abstract} In this paper we study the topological invariant $\tc(X)$ reflecting the complexity of algorithms for autonomous robot motion.
Here, $X$ stands for the configuration space of a system and $\tc(X)$ is, roughly, the minimal number of continuous rules which are needed to construct a motion planning algorithm in $X$. We focus on the case when the space $X$ is aspherical; then the number $\tc(X)$ depends only on the fundamental group 
$\pi=\pi_1(X)$ and we denote it $\tc(\pi)$. We prove that $\tc(\pi)$ can be characterised as the smallest integer $k$ such that the canonical $\pi\times\pi$-equivariant map
of classifying spaces $$E(\pi\times\pi) \to E_\D(\pi\times\pi)$$ can be equivariantly deformed into the $k$-dimensional skeleton of $E_\D(\pi\times\pi)$. 
The symbol $E(\pi\times\pi)$  denotes the classifying space for free actions and  $E_\D(\pi\times\pi)$ denotes the classifying space for actions with isotropy in a certain family $\D$ of subgroups of $\pi\times\pi$. 
Using this result we show how one can estimate $\tc(\pi)$ in terms of the equivariant Bredon cohomology theory. We prove that 
$\tc(\pi) \le  \max\{3, \cd_\D(\pi\times\pi)\},$ where
$\cd_\D(\pi\times\pi)$ denotes the cohomological dimension of $\pi\times\pi$ with respect to the family of subgroups $\D$. We also introduce a Bredon cohomology refinement of the canonical class and prove its universality. Finally we show that for a large class of {\it principal} groups 
(which includes all torsion free hyperbolic groups as well as all torsion free nilpotent groups)
the essential cohomology classes in the sense of Farber and Mescher \cite{FM} are exactly the classes having Bredon cohomology extensions with respect to the family $\D$. 
\end{abstract}

%
%
%

\section{Introduction} 

{\it The topological complexity}, 
$\tc(X)$, is a numerical 
homotopy invariant of a path-connec\-ted topological space $X$, originally introduced in \cite{Far03} (see also \cite{Finv})
which is motivated by the motion planning problem of robotics. Roughly, $\tc(X)$ is the minimal number of continuous rules which are needed to construct an algorithm for autonomous motion planning of a system having $X$ as its configuration space.

To give more detail, assume that a system (robot) has to be programmed to move autonomously from any initial state to any final state. 
Let $X$ denote the configuration space of the system; points of $X$ represent states of the system and 
continuous paths 
in $X$ represent 
motions of the system.
{\it A motion planning algorithm} is a function which associates with any pair of states $(A, B)\in X\times X$ a continuous motion of the system 
starting at $A$ and ending at $B$. In other words, a motion planning algorithm is 
a section of the path fibration 
\begin{eqnarray}\label{fibration}
p: X^I\to X\times X, \quad p(\gamma) = (\gamma(0), \gamma(1)).
\end{eqnarray}
Here $X^I$ denotes the space of all continuous paths $\gamma: I=[0, 1]\to X$ equipped with the compact-open topology. 
It is easy to see that the fibration (\ref{fibration}) admits a continuous section if and only if $X$ is contractible \cite{Far03}.
{\it The topological complexity} $\tc(X)$ is an integer (see Definition \ref{def1} below) reflecting the complexity of 
this fibration. It has several different characterisations, see \cite{Far06}. 
Intuitively, $\tc(X)$ is a measure of the navigational complexity of $X$ viewed as the configuration space of a system. 
$\tc(X)$ is similar in spirit to the classical Lusternik - Schnirelmann category $\cat(X)$.
The invariants $\tc(X)$ and $\cat(X)$ are special cases of a more general notion of genus of a fibration introduced by A. Schwarz \cite{Sv66}. 
A recent survey of 
the concept $\tc(X)$ and robot motion planning algorithms in practically interesting configuration spaces can be found in \cite{Frecent}.

One of the main properties of $\tc(X)$ is its {\it homotopy invariance} \cite{Far03}, i.e. $\tc(X)$ depends only on the homotopy type of $X$. 
This property is helpful for the task of computing $\tc(X)$ in various examples since cohomological tools can be employed. 
In the case when the configuration space $X$ is {\it aspherical}, i.e. $\pi_i(X)=0$ for all $i>1$, the number $\tc(X)$ depends only on the fundamental group 
$\pi=\pi_1(X)$ and it was observed in \cite{Far06} that one should to be able to express $\tc(X)$ in terms of the algebraic properties of the group
$\pi$ alone. This remark justifies the notation $\tc(\pi)$ for $\tc(K(\pi,1))$. 

A similar question for the Lusternik - Schnirelmann category $\cat(X)$ was solved by S. Eilenberg and T. Ganea in 1957 in the seminal paper \cite{EG}. 
Their theorem relates $\cat(X)$ and the cohomological dimension of the fundamental group $\pi$ of $X$. 

The problem of computing $\tc(\pi)$ as an algebraic invariant of the group $\pi$ has attracted the attention of many mathematicians and many interesting partial results have been obtained. It is easy to see that $\tc(\pi)=\infty$ if $\pi$ has torsion; therefore we shall always restrict our attention to torsion free groups $\pi$.

The initial papers \cite{Far03}, \cite{Far06} contained computations of $\tc(X)$ for graphs, closed orientable surfaces and tori. In \cite{FarberYuzvinsky} the number $\tc(X)$ was computed for the case when $X$ is the configuration space of many particles moving on the plane without collisions. 
D. Cohen and G. Pruidze  \cite{CohenPruidze} calculated the topological complexity of complements of general position arrangements and Eilenberg -- MacLane spaces associated to certain right-angled Artin groups.

In a recent breakthrough, the topological complexity of closed non-orientable surfaces of genus $g \geq 2$ was computed by A. Dranishnikov for $g \geq 4$ in \cite{Dranish} and by D. Cohen and L. Vandembroucq for $g=2, 3$ in \cite{CohenVandem}. In both these articles it is shown that $\tc(\pi)$ attains its maximum, i.e. coincides with $\mathrm{cd}(\pi \times \pi)$, the cohomological dimension of the group $\pi\times\pi$. 

The estimates of M. Grant \cite{Grant} give good upper bounds for $\tc(\pi)$ for nilpotent fundamental groups $\pi$. 
In  \cite{GrantLuptonOprea}, M. Grant, G. Lupton and J. Oprea proved that $\tc(\pi)$ is bounded below by the cohomological dimension 
of $A \times B$ where $A$ and $B$ are  subgroups of $\pi$ whose conjugates intersect trivially. Using these estimates, M. Grant and D. Recio-Mitter \cite{GrantRecio}
have computed $\tc(\pi)$ for certain subgroups of Artin's braid groups. 

Y. Rudyak \cite{Rudyak} went in the opposite direction by showing that for any pair of positive integers $k, \ell$ satisfying $k\le \ell \le 2k$ there exists a finitely presented group $\pi$ such that $\cd(\pi)=k$ and $\tc(\pi)= \ell.$ 

In a recent preprint \cite{FM} M. Farber and S. Mescher showed that for a large class of groups (including all torsion free hyperbolic groups) the topological complexity $\tc(\pi)$ equals either $\cd(\pi\times\pi)$ or $\cd(\pi\times\pi)-1$. Since hyperbolic groups are typical in many models of random groups this gives an answer with possible error 1 for a typical group. Note that $\cd(\pi\times\pi)$ is obviously an upper bound for $\tc(\pi)$ 
for any $\pi$.  

In this paper we tackle the general problem of understanding $\tc(\pi)$ from a different direction, using the tools of equivariant topology. 
We are not interested in computing examples, but rather in reformulating the problem itself so that interactions with subjects such as group theory and homological algebra become apparent. This re-interpretation, together with previously computed examples of $\tc(\pi)$, provides illustrative examples for Bredon cohomology with respect to a family of subgroups.

Firstly we reduce the problem to a question about classifying spaces of families of subgroups. Namely, we define a special class $\D$ of subgroups of $G=\pi\times\pi$ and prove that the number 
$\tc(\pi)$ coincides with the smallest $k$ such that the canonical map of classifying spaces $$E(G)\to E_\D(G)$$ can be factored through 
a $G$-CW-complex of dimension $\le k$. Here $E(G)$ is the classical classifying space for free $G$-actions and $E_\D(G)$ is the classifying space for $G$-actions with isotropy subgroups in the class $\D$. 
Using this reduction we establish an upper bound
\begin{eqnarray}
\tc(\pi)\le \max\{3, \cd_\D(\pi\times\pi)\}.
\end{eqnarray}
where $\cd_\D(\pi\times\pi)$ denotes the cohomological dimension of $\pi\times\pi$ with respect to the family $\D$. 
Secondly, we use Bredon cohomology to produce lower bounds for $\tc(\pi)$. Namely we show that if 
(for some $\od$-module $\um$)
there exists a Bredon cohomology class 
$$\underline \alpha\in H^n_\D(\pi\times\pi, \um)$$
such that the cohomology class 
$$\Phi(\underline \alpha)\not=0\in H^n(\pi\times\pi, M)$$
is nonzero, 
then $\tc(X) \ge n$. Here $M$ denotes the principal component of $\um$ and $$\Phi: H^n_\D(\pi\times\pi, \um) \to H^n(\pi\times\pi, M)$$ is a natural homomorphism from the Bredon cohomology to the usual twisted cohomology. The notions we use here are explained in full detail in the sequel.

We define a \emph{Bredon cohomology generalisation of the canonical class} $\uu\in H^1_\D(\pi\times\pi;\underline{I})$  which refines the canonical class $\vv\in H^1(\pi\times\pi;I)$ introduced in \cite{CosFar}. We prove a universality theorem for
the powers of the Bredon canonical class $\uu$ which implies, in particular, that $\cd_\D(\pi\times\pi) = {\rm {height}}(\uu)$ (the \emph{height} of a cohomology class is the exponent of the maximal non-vanishing power). 

In \cite{FM}, Farber and Mescher introduced the notion of {\it an essential cohomology class} as a class $\beta\in H^n(\pi\times\pi, A)$ which can be obtained via a coefficient homomorphism $\mu: I^n\to A$ from the power $\vv^n$ of the canonical class, i.e. $\beta=\mu_\ast(\vv^n)$. 
In this paper we introduce 
a class of {\it principal groups} $\pi$ and we show that for principal groups a cohomology class in $H^n(\pi\times\pi;M)$ is in the image of the map $\Phi$ from Bredon cohomology if and only if it is essential. We also prove that the class of principal groups includes all torsion free hyperbolic groups and all torsion free nilpotent groups.

Curiously, the fundamental group of the Klein bottle is not principal (see Example \ref{Klein}) but nevertheless for this group  
$$\tc(\pi) = {\rm {height}}(\vv) $$
as follows from the theorem of D. Cohen and L. Vandembroucq \cite{CohenVandem}. 

While results of this paper are more conclusive for $\tc(\pi)\ge 3$,  
we mention that $\Z$ is the only group satisfying $\tc(\pi)=1$ (as follows from \cite{GLO1}) and 
 groups with $\tc(\pi)=2$ are likely quite restricted, see \cite{BR}.
The obvious examples of groups $\pi$ with $\tc(\pi)=2$ include $\Z^2$ and the non-commutative free group $F$.

%

\section{The first reduction}

\subsection{The concept of topological complexity}

We start by recalling the definition of the invariant $\tc(X)$. 

\begin{definition}\label{def1}{\rm
Given a path-connected topological space $X$, the topological complexity of $X$ is the minimal integer 
$\tc(X)=k$ such that the Cartesian product $X \times X$ can be covered by $k+1$ open subsets
$$X \times X = U_0 \cup U_1 \cup \cdots \cup U_k$$ with the property that for any $i = 0, 1,2,\dots ,k$ there exists a continuous section of the fibration
(\ref{fibration})
$$s_i: U_i \to X^I,\quad
p\circ s_i={\rm incl}_{U_i}$$ over $U_i$. If no such $k$ exists we will set $\tc(X)=\infty$.
}
\end{definition}

Note that in this paper we are using {\it the reduced version} of $\tc(X)$ which is one less than the original notion used in \cite{Far03}, \cite{Finv}, \cite{Frecent}
and \cite{FM}. 

For convenience of the reader we also recall the notion of the Schwarz genus of a fibration (also known as {\it the sectional category}). 
\begin{definition}\label{def12}{\rm
Let $p: E\to B$ be a Serre fibration over a path-connected topological space $B$. The Schwarz genus of $p$ is defined as the smallest integer $k$ such that
the base $B$ admits an open cover $B= U_0\cup U_1\cup \cdots \cup U_k$ such that the fibration $p$ admits a continuous section over $U_i$ 
for each $i=0, 1, \dots, k$. 
}
\end{definition}


This paper is mainly dedicated to the problem of computing $\tc(X)$  in the case when $X$ is an aspherical finite cell complex. 
Recall that a connected cell complex $X$ is said to be {\it aspherical} if $\pi_i(X)=0$ for all $i>1$. The notation $X=K(\pi, 1)$ means that 
$X$ is aspherical and its fundamental group is $\pi$. 
A key property of $\tc(X)$ is its homotopy invariance, see  \cite{Far03}.  
The homotopy invariance of the topological complexity implies that the number $$\tc(\pi)=\tc(K(\pi, 1))$$ depends only on the group $\pi$. 

\subsection{} Many systems of practical interest have aspherical configuration spaces. 
Consider for example the problem of coordinated collision free motion planning of a set of objects on the plane $\R^2$. 
We may represent the objects by 
discs of radius $r>0$ and the state of each disc is determined by the position of its centre $A_i\in \R^2$ where $i=1, \dots, n$. 
Thus a state of the system 
is a configuration of points $(A_1, A_2, \dots, A_n)$, where $A_i\in \R^2$, such that $$|A_i-A_j|>2r, \quad i\not= j.$$ Let $F_r(\R^2, n)$ denote the configuration space of this system. 
It is common to relax the problem and consider instead the weaker condition $$A_i\not=A_j, \quad i\not= j$$ which leads to the usual configuration space 
$F(\R^2, n)$ of $n$ distinct points on the plane. It is easy to see that $F_r(\R^2, n)$ and $F(\R^2, n)$ are homeomorphic and, moreover, it is well known that the space $F(\R^2, n)$ is aspherical as can be seen using the tower of Fadell-Neuwirth fibrations.

%
%
%
%
%
%
%

\subsection{} \label{remarks1} Consider a continuous partial section  $s: U\to X^I$ of the fibration (\ref{fibration}) over a subset $U\subset X\times X$.
Using the exponential correspondence, the map $s$ can be viewed as a homotopy $h:U\times I\to X$ where $h(u, t)=s(u)(t)$ for $u\in U, t\in I$. 
Let $p_j: X\times X\to X$ (where $j=1,2$) denote the projections onto the first and the second factors. 
The property of $s$ being a section can be expressed by saying that the homotopy $h$ connects the projections of  
$U$ onto the first and second coordinates, i.e. $h(u, 0)=p_1(u)$ and $h(u, 1)=p_2(u)$. 

Thus we see that the open sets $U_i\subset X\times X$ which appear in Definition \ref{def1} can be equivalently characterised by the property that their two projections $U_i\to X$ on the first and the second factors are homotopic. 

In the case when the space $X$ is aspherical we can use the following property: 
For a connected space $U$ that is homotopy equivalent to a cell complex, the set of homotopy classes of maps $U\to X$ is in a one-to-one correspondence 
with the set of conjugacy classes of homomorphisms $\pi_1(U, u_0)\to \pi_1(X, x_0)$, see Chapter V, Corollary 4.4 in \cite{Whi}.
Recall that two group homomorphisms $f, g: \pi_1(U, u_0)\to \pi_1(X, x_0)$ are 
conjugate if there exists $\beta\in \pi_1(X, x_0)$ such that for all $\alpha\in \pi_1(U, u_0)$ one has 
$f(\alpha)=\beta g(\alpha)\beta^{-1}$. 

These remarks lead to the following definition:

\begin{definition}\label{def2}{\rm
Let $X$ be a path-connected topological space with fundamental group $\pi=\pi_1(X, x_0)$. {\it The $\D$-topological complexity}, $\tc^\D(X)$, is defined as the minimal number 
$k$ such that $X \times X$ can be covered by $k+1$ open subsets
$X \times X = U_0 \cup U_1 \cup \cdots \cup U_k$ with the property that for any $i = 0, 1,2,\dots ,k$ 
and for every choice of the base point $u_i\in U_i$ the homomorphism $\pi_1(U_i, u_i)\to \pi_1(X\times X, u_i)$ induced by 
the inclusion $U_i\to X\times X$ takes values in a subgroup conjugate to the diagonal $\Delta\subset \pi\times\pi$. 
}
\end{definition}

Recall that 
there is an isomorphism $\pi_1(X\times X, u_i) \to \pi_1(X\times X, (x_0, x_0))=\pi\times \pi$ determined uniquely up to conjugation, and 
the diagonal inclusion $X\to X\times X$ induces the inclusion $\pi\to \pi\times\pi$ onto the diagonal $\Delta$. 

\begin{lemma}\label{lm3}
One has $\tc^\D(X)=\tc(X)$ if $X$ is a finite aspherical cell complex.
\end{lemma}
\begin{proof}
It follows from the remarks given in \S \ref{remarks1}. Here we use the known fact that an open subset of a finite CW-complex is homotopy equivalent to a countable CW-complex. Indeed, by Theorem 1 of J. Milnor \cite{Milnor}, a space is homotopy equivalent to a countable CW-complex 
if and only if it  is homotopy equivalent to an absolute neighbourhood retract (ANR). Any finite CW-complex is an ANR and an open subset of an ANR is an ANR. Thus, an open subset of a finite CW-complex is an ANR and hence has the homotopy type of a countable CW-complex. 
\end{proof}

\begin{lemma}\label{lm4}
Let $X$ be a finite aspherical cell complex with fundamental group $\pi=\pi_1(X, x_0)$.
Let 
$q: \widehat{X\times X}\to X\times X$ be the connected covering space
corresponding to the diagonal subgroup $$ \Delta\subset \pi\times\pi=\pi_1(X\times X, (x_0, x_0)).$$ 
Then 
the $\D$-topological complexity $\tc^\D(X)$ coincides with the Schwarz genus of $q$. 
\end{lemma}
\begin{proof}
For an open subset $U\subset X\times X$, the condition that the induced map $\pi_1(U, u) \to \pi_1(X\times X, u)$ takes values in a subgroup 
conjugate to the diagonal $\Delta$ is equivalent to the condition that $q$ admits a continuous section over $U$. 
Using this remark the Lemma follows by comparing the definitions of $\tc^\D(X)$ and of Schwarz genus. 
\end{proof}

\begin{remark}{\rm 
If we remove the assumption that $X$ is aspherical then the topological complexity $\tc(X)$ is greater than or equal to the Schwarz genus of $q$, see
\cite{FTY}, Theorem 4.1. 
}
\end{remark}

Next we introduce terminology and notations which will be used in the statement of Theorem \ref{thm0}.

\subsection{} Recall that the join $X\ast Y$ of topological spaces $X$ and $Y$ can be defined as the quotient of the product $X\times [0,1]\times Y$ with respect to the equivalence relation $(x, 0, y)\sim (x, 0, y')$ and $(x, 1, y)\sim (x', 1, y)$ for all $x, x'\in X$ and $y, y'\in Y$. We have an obvious embedding 
$X\to X\ast Y$ given by $x\mapsto (x, 0, y)$ where $y\in Y$ is arbitrary. 

One may use the following notation. A point $(x, t, y)\in X\times [0,1]\times Y/\sim$ \, can be written as a formal linear combination $(1-t)x+ty$. This notation is clearly consistent with the identifications of the join. 

Let $\Delta^k$ denote the standard $k$-dimensional simplex, i.e. $$\Delta^k=\{(t_0, t_1, \dots, t_k); t_i\ge 0, \quad \sum_{i=0}^kt_i=1\}.$$  We may define 
the multiple join
$X_0\ast X_1\ast\dots\ast X_k$ of topological spaces $X_0, \dots, X_k$ as the quotient of the product $(\prod_{i=0}^k X_i)\times \Delta^k$ with respect to an equivalence relation $\sim$ described below. 
The points of the join are written as formal linear combinations 
$$x=t_0x_0+t_1x_1+\dots+t_kx_k, \quad x_i\in X_i, \quad (t_0, t_1, \dots, t_k)\in \Delta^k,$$
and we say that $x\sim x'$ where $x'=t'_0x'_0+t'_1x'_1+\dots+t'_kx'_k$ iff $t_i=t'_i$ for all $i=0, \dots, k$ and $x_i=x'_i$ provided $t_i\not=0$. 

\subsection{} \label{sec25} Let $\pi$ be a discrete group. We shall view $\pi$  as a discrete topological space with the following left $\pi\times\pi$-action: 
\begin{eqnarray}\label{action}(x, y)\cdot g=xgy^{-1}.\end{eqnarray}
This action is transitive and the isotropy subgroup of the unit element $1\in \pi$ coincides with the diagonal subgroup $\Delta\subset \pi\times \pi$. 
The isotropy subgroups of the other elements are the conjugates of $\Delta$. 

\subsection{} For an integer $k\ge 0$, let $E_k(\pi)$ denote the $(k+1)$-fold join 
$$E_k(\pi) = \pi\ast\pi\ast \dots\ast \pi.$$
We shall equip $E_k(\pi)$ with the left diagonal $\pi\times\pi$-action determined by the $\pi\times\pi$-action on $\pi$ as in \S \ref{sec25} above. 
Each $E_k(\pi)$ is naturally a $k$-dimensional equivariant simplicial complex with $k$-dimensional simplexes in 1-1 correspondence with sequences $(g_0, g_1, \dots, g_k)$ of group elements
$g_i\in \pi$. 
Note that $E_k(\pi)$ is $(k-1)$-connected and is in fact homotopy equivalent to a wedge of $k$-dimensional spheres. 

\subsection{} There is a natural equivariant embedding $$E_k(\pi) \hookrightarrow E_{k+1}(\pi)=E_k(\pi)\ast \pi.$$ Using it we may define the simplicial complex 
$$E(\pi)= \bigcup_{k=0}^\infty E_k(\pi) = \pi\ast \pi\ast \pi\ast \dots,$$
the join of infinitely many copies of $\pi$. 

\subsection{} Furthermore, let $E(\pi\times\pi)$ denote the classical classifying space for free $\pi\times\pi$ actions, i.e. 
$$E(\pi\times\pi)= (\pi\times\pi)\ast (\pi\times\pi)\ast \dots, $$
the join of infinitely many copies of $\pi\times\pi$. We shall view each copy of $\pi\times \pi$ as a discrete topological space with the left free action
of $\pi\times\pi$ given by $(x, y)\cdot(g, h)=(xg, yh)$ for $x, y, g, h\in \pi$. 
The space $E(\pi\times\pi)$ inherits the diagonal action of the group $\pi\times\pi$. 

\subsection{} The map $F: \pi\times\pi\to \pi$ given by $F(x, y)=xy^{-1}$ is $\pi\times\pi$-equivariant. The natural extension of $F$ to the infinite joins 
defines a $\pi\times\pi$-equivariant map
\begin{eqnarray}\label{F}
F: E(\pi\times\pi) \to E(\pi).
\end{eqnarray}

\begin{theorem}\label{thm0} Let $X$ be a finite aspherical cell complex and let $\pi=\pi_1(X, x_0)$ be its fundamental group. 
Then $\tc(X)$ coincides with the smallest integer $k$ such that there exists a $\pi\times\pi$-equivariant map $E(\pi\times\pi)\to E_k(\pi)$. 
%
\end{theorem}
\begin{proof} Let $p: \tilde X\to X$ denote the universal cover of $X$. Here $\tilde X$ is an equivariant cell complex with free left $\pi$-action. 
The map $p\times p: \tilde X\times\tilde X\to X\times X$ is the universal cover of $X\times X$. We shall view $p\times p$ as a principal 
$G=\pi\times\pi$-bundle and for $k=0, 1, \dots$ construct the associated bundle 
\begin{eqnarray}\label{qk}
q_k: (\tilde X\times \tilde X)\times_{G} E_k(\pi)\to X\times X.
\end{eqnarray}
Here $(\tilde X\times \tilde X)\times_{G} E_k(\pi)$ denotes the quotient of the product $(\tilde X\times \tilde X)\times E_k(\pi)$ with respect to the 
following $G=\pi\times\pi$-action: $(g, h)\cdot (x, x', z) = (gx, hx', (g, h)\cdot z)$ where $g, h\in \pi$, $x, x'\in \tilde X$ and $z\in E_k(\pi)$. 

First we observe that the fibration $q_0$ coincides with the covering space $q: \widehat{X\times X}\to X\times X$ corresponding to the diagonal subgroup 
$\Delta\subset \pi\times\pi$ which appears in Lemma \ref{lm4}. Indeed, $E_0(\pi)=\pi$ has a transitive $G=\pi\times\pi$-action and the isotropy of the unit element $1\in \pi$ is the diagonal $\Delta\subset G=\pi\times\pi$. Hence we obtain a homeomorphism 
$$(\tilde X\times \tilde X\times E_0(\pi))/G \to (\tilde X\times \tilde X)/\Delta$$
commuting with the projections onto $X\times X$; thus we see that the fibration $q_0$ is isomorphic to the fibration
$p\times p: (\tilde X\times \tilde X)/\Delta \to X\times X$. It is obvious that the latter fibration 
is isomorphic to the connected covering $q$
corresponding to the diagonal subgroup $\Delta\subset \pi\times\pi$. 

Applying Lemma \ref{lm3} and Lemma \ref{lm4} we obtain that $\tc(X)$ coincides with the Schwarz genus of the fibration $q_0$. 

Next we apply a theorem of A. Schwarz (see \cite{Sv66}, Theorem 3) stating that genus of a fibration $p: E\to B$ equals the smallest integer $k$ such that 
the  
fiberwise join $p\ast p\ast \dots \ast p$ of $k+1$ copies of the fibration $p: E\to B$ admits a continuous section. The fiberwise join of $k+1$ copies of the 
fibration $q_0$ coincides with the fibration $q_k$. Thus we obtain that $\tc(X)$ coincides with the smallest $k$ such that $q_k$ has a continuous section. 

Finally we apply Theorem 8.1 from \cite{Hue}, chapter 4,  which states that continuous sections of the fibre bundle $q_k$ are in 1-1 correspondence with 
$G=\pi\times\pi$-equivariant maps 
\begin{eqnarray}\label{finally}\tilde X\times \tilde X \to E_k(\pi).\end{eqnarray}
Thus, $\tc(X)$ is the smallest $k$ such that a $G=\pi\times\pi$-equivariant map (\ref{finally}) exists. 
Finally we observe that the space $\tilde X\times \tilde X$ is $G=\pi\times\pi$-equivariantly homotopy equivalent to $E(\pi\times\pi)$ (in view of the Milnor construction) and the result follows. 
\end{proof}

\section{The second reduction}

In this section we prove the following statement which gives an intrinsic version of Theorem \ref{thm0}.

\begin{theorem}\label{thm00} Let $X$ be a finite aspherical cell complex and let $\pi=\pi_1(X, x_0)$ be its fundamental group. 
Then $\tc(X)$ coincides with the minimal dimension of a $\pi\times\pi$-CW complex $L$ such that 
the map $F$ (see (\ref{F})) can be factored as follows:
\begin{eqnarray}\label{FF}
E(\pi\times\pi) \to L \to E(\pi).
\end{eqnarray}
\end{theorem}

The proof of Theorem \ref{thm00} will follow a brief review of the basic material concerning classifying spaces and families of subgroups; we shall mainly follow \cite{Lue}. 

\subsection{} Let $G$ be a discrete group. {\it A $G$-CW-complex} is a CW-complex $X$ with a left $G$-action such that for each open cell
$e\subset X$ and each $g\in G$ with $ge\cap e\not=\emptyset$, the left multiplication by $g$ acts identically on $e$. 

A simplicial complex with a simplicial $G$-action is a $G$-CW-complex (with respect to the barycentric subdivision), see \cite{Lue}, Example 1.5.

{\it A family $\F$ of subgroups} of $G$ is a set of subgroups of $G$ which is closed under conjugation and finite intersections. 

\subsection{} \label{22}{\it A classifying $G$-CW-complex} $E_\F(G)$ with respect to a family $\F$ of $G$ is defined as a $G$-CW-complex $E_\F(G)$ such that 
\begin{itemize}\item[{(a)}] the isotropy subgroup of any element of $E_\F(G)$ belongs to $\F$; \item[{(b)}] For any $G$-CW-complex $Y$ all of whose
isotropy subgroups belong to $\F$ there is up to $G$-homotopy 
exactly one $G$-map $Y\to E_\F(G)$. 
\end{itemize}

A $G$-CW-complex $X$ is a model for $E_\F(G)$ if and only if all its isotropy subgroups belong to the family $\F$ and for each $H\in \F$ the set 
of $H$-fixed points $X^H$ is weakly contractible, i.e. $\pi_i(X^H, x_0)=0$ for any $i=0, 1, \dots$ and for any $x_0\in X^H$. See \cite{Lue}, Theorem 1.9. 

\subsection{} We shall use below the equivariant version of the Whitehead Theorem which we shall state as follows (see \cite{May}, Theorem 3.2 in Chapter 1). 

\begin{theorem}[Whitehead theorem]\label{thwhite} Let $f: Y\to Z$ be a $G$-map between $G$-CW-complexes such that for each subgroup $H\subset G$ the induced map 
$\pi_i(Y^H, x_0)\to \pi_i(Z^H, f(x_0))$ is an isomorphism for $i<k$ and an epimorphism for $i=k$ for any base point $x_0\in Y^H$. Then for any $G$-CW-complex $X$ the induced map 
on the set of $G$-homotopy classes 
$$f_\ast: [X, Y]_G \to [X, Z]_G $$
is an isomorphism if $\dim X<k$ and an epimorphism if $\dim X\le k$. 
\end{theorem}

\subsection{Proof of Theorem \ref{thm00}}\label{sec24} First note that $G=\pi\times\pi$ acts freely on $E(\pi\times\pi)$ which is the classifying $G$-CW-complex for free $G$-actions (the Milnor construction). We refer to Example 1.5 from \cite{Lue} which implies that $E(\pi\times\pi)$ is a $G$-CW-complex. 

Next we examine the isotropy subgroups of $G=\pi\times\pi$ acting on $E_k(\pi)$ and $E(\pi)$. 
Recall that $G$ acts on $\pi$ according to formula (\ref{action}). The isotropy of an element $g\in \pi$ is 
the subgroup $\{(a, g^{-1}ag); a\in \pi\}\subset \pi\times\pi$ which is conjugate to the diagonal subgroup $\Delta$. 

It is easy to see that for a subgroup $H\subset G$ the fixed point set $\pi^H$ is non-empty iff $H$ is contained in a subgroup 
 conjugate to the diagonal $\Delta\subset G$.

For an element $x\in E_k(\pi)$, 
$$x= t_0x_0+t_1x_1+\dots+t_kx_k,$$
where $x_i\in \pi$, $t_i\in (0,1]$, $i=0, 1, \dots, k$, $t_0+t_1+\dots+t_k=1$, 
the isotropy subgroup is the intersection of the isotropy subgroups of the elements $x_i$. This intersection can be presented as follows.
Let $S$ denote the set $\{x_ix_j^{-1}\, ;\,  i, j =0, 1, \dots, k\}$. 
The symbol $Z(S)$ denotes the centraliser of $S$, i.e. the set of all $a\in \pi$ which commute with any element of $S$. 
Then the isotropy subgroup of $x$
equals 
\begin{eqnarray}\label{hbs}
H_{b, S}=\{(a, bab^{-1}); a\in Z(S)\}
\end{eqnarray}
where $b=x_i^{-1}$ for any $i=0, 1, \dots, k$. 


If $H\subset \pi\times\pi$ is a subgroup contained in a subgroup of type (\ref{hbs}), i.e. $H\subset H_{b, S}$, 
then the set $\pi^H$ is not empty and
$$E_k(\pi)^H = \pi^H \ast \pi^H\ast \dots \ast \pi^H, \quad\quad (k+1 \quad  \mbox{times}).$$
We see that the space $E_k(\pi)^H$ is nonempty and is $(k-1)$-connected. At the same time the space 
$E(\pi)^H=\pi^H\ast\pi^H\ast\cdots$ (the infinite join) is non-empty and contractible. We will use this property below in order to invoke the Whitehead theorem. 

We denote by $\D$ the family of subgroups of $\pi\times\pi$ containing the trivial subgroup and the groups $H_{b, S}$, for all $b\in \pi$ and 
all finite subsets $S\subset \pi$. 

The above discussion shows that $E(\pi)$ is the classifying $G$-CW-complex $E_\D(G)$ with respect to the family $\D$, see \S \ref{22}. 
In particular, we obtain that any two $G$-maps $X\to E(\pi)$ are $G$-homotopic provided all isotropy subgroups of $X$ are in $\D$. 

Let $k_1$ denote the minimal $k$ such that there exists an equivariant map $E(\pi\times\pi)\to E_k(\pi)$. 
We know that $k_1=\tc(X)$ by Theorem \ref{thm0}. 
Let $k_2$ be the smallest dimension of a 
$G$-CW complex $L$ admitting a factorisation (\ref{FF}). 
We have $k_2\le k_1$ since $\dim E_k(\pi)=k$ and any two equivariant maps $E(G)\to E(\pi)$ are equivariantly homotopic. 
On the other hand, suppose we have $$E(\pi\times\pi) \stackrel\alpha\to L \stackrel \beta\to E(\pi)$$ 
with $\dim L\le k$. We may apply the Whitehead Theorem \ref{thwhite} to the inclusion 
$E_k(\pi)\to E(\pi)$ concluding that for any $G$-CW-complex $L$ of dimension $\le k$ the map $$[L, E_k(\pi)]_G\to [L, E(\pi)]_G$$ is surjective. 
We then obtain a $G$-map
$g: L\to E_k(\pi)$ and its composition $g\circ \alpha: E(\pi\times\pi)\to E_k(\pi)$; clearly the composition 
$E(\pi\times\pi)\stackrel{g\circ \alpha}\to E_k(\pi) \hookrightarrow E(\pi)$ is $G$-homotopic to $\beta\circ\alpha$. 
This shows that $k_1\le k_2$ and hence $k_1=k_2$ proving Theorem \ref{thm00}. \qed

We can restate Theorem \ref{thm00} as follows:

\begin{theorem}\label{thm000} Let $X$ be a finite aspherical cell complex and let $\pi=\pi_1(X, x_0)$ be its fundamental group. 
Let $G$ denote the group $\pi\times\pi$. Then $\tc(X)$ coincides with the minimal integer $k$ such that the canonical map 
\begin{eqnarray}\label{eq}
E(G) \to E_\D(G)
\end{eqnarray}
is $G$-equivariantly homotopic to a map with values in the $k$-dimensional skeleton $E_\D(G)^{(k)}$. 
\end{theorem}
\begin{proof}
If the map (\ref{eq}) is $G$-homotopic to a map with values in $E_\D(G)^{(k)}$ then we can take $L=E_\D(G)^{(k)}$ to obtain a factorisation of Theorem \ref{thm00}. Conversely, given a factorisation of Theorem \ref{thm00}, the map $L\to E_\D(G)$ can be deformed into $E_\D(G)^{(k)}$ using the $G$-cellular 
approximation theorem. 
\end{proof}

Let us recall that the Lusternik - Schnirelmann category of an aspherical space can be characterised in a similar way:

\begin{proposition}\label{propeg}
Let $X$ be a finite aspherical cell complex and let $\pi=\pi_1(X, x_0)$ be its fundamental group. 
Then the Lusternik - Schnirelmann category $\cat(X)$ coincides with the minimal dimension of a $\pi$-CW complex $L$ such that 
the identity map $E(\pi)\to E(\pi)$ can be $\pi$-equivariantly factored as follows
\begin{eqnarray}\label{categ}
E(\pi) \to L \to E(\pi).
\end{eqnarray}
\end{proposition}

This statement is essentially contained in \cite{EG}, compare \cite[Proposition 1]{ EG} where, however there is an assumption $n\ge 2$. 
The proof of Proposition \ref{propeg} in the general case can be obtained similarly to the proof of Theorem \ref{thm000} and we shall briefly indicate the main steps. Firstly, one states that $\cat(X)$ equals the Schwarz genus of the universal covering $\tilde X\to X$, compare Lemma \ref{lm3} and Lemma \ref{lm4}.  
Secondly, using the theorem of Schwarz about joins we obtain that $\cat(X)$ equals the smallest $k$ such that the fibration 
$\tilde X\times_\pi E_k(\pi)\to X$ admits a continuous section, compare the proof of Theorem \ref{thm0}. 
Here we view the complex $E_k(\pi)$  with the left $\pi$-action which is free. Thirdly, we find that $\cat(X)$ equals the smallest $k$ such that there exists a $\pi$-equivariant map $\tilde X\to E_k(\pi)$, compare Theorem \ref{thm0}. And finally, one uses the universal properties of the classifying space 
$E(\pi)=\tilde X$ and the equivariant Whitehead theorem to restate the result in the form of Proposition \ref{propeg}.

\subsection{} Let $\od$ denote the orbit category with respect to the family $\D$, see \cite{Bre}; we shall recall these notions in the following section.
Let $\cd_\D(\pi\times\pi)$ denote the cohomological dimension of the constant $\od$-module $\uz$. 
Since $E(\pi)$ is a model for the classifying space $E_\D(G)$, applying Theorem 5.2 from \cite{Lue} we obtain that $E(\pi)$ has 
the equivariant homotopy type of a $G$-CW-complex of dimension $\le \max\{3, \cd_\D(\pi\times\pi)\}$. Together with Theorem \ref{thm00} this gives the following Corollary. 

\begin{corollary}\label{upper}
Let $X$ be a finite aspherical cell complex and let $\pi=\pi_1(X, x_0)$ be its fundamental group. 
Then 
\begin{eqnarray}
\tc(X)\le \max\{3, \cd_\D(\pi\times\pi)\}.\end{eqnarray}
\end{corollary}


\begin{proposition}\label{Shapiro}
For any discrete group $\pi$ we have $\mathrm{cd}(\pi)\leq \mathrm{cd}_\mathcal{D}(\pi\times \pi)$.
\end{proposition}

\begin{proof}
Recall that $G$ denotes $\pi\times\pi$. Assume first that $k:=\mathrm{cd}_\mathcal{D}(G)\ge 3$, so that there exists a $k$-dimensional model for $E_\mathcal{D}(G)$. Since the trivial subgroup is in $\mathcal{D}$, the space $E_\mathcal{D}(G)$ is contractible. Restricting the $G$-action to the subgroup $\pi\times 1\subseteq G$ gives a free $\pi$-action, since $E_\mathcal{D}(G)$ has isotropy in $\mathcal{D}$ and $(\pi\times 1)\cap H$ is trivial for all $H\in \mathcal{D}$. Hence $E_\mathcal{D}(G)$ is a $k$-dimensional model for $E(\pi)$, and it follows that $\mathrm{cd}(\pi)\leq k$.

The general algebraic result of Proposition \ref{Shapiro} follows from Shapiro's lemma in Bredon cohomology \cite[Proposition 3.31]{Fl}, which gives isomorphisms
\[
H^*(\pi;M)\cong H^*_\mathcal{D}(G;\operatorname{coind}_I(M))
\]
for each $\pi$-module $M$. Here the co-induction is along the inclusion functor $I:\mathcal{O}_{\{1\}}(\pi\times 1)\to \mathcal{O}_\mathcal{D}(G)$. This argument does not require the assumption $k\ge 3$. 
\end{proof}

\begin{corollary} Suppose that $\pi = \Bbb Z^k$. Then 
$\mathrm{cd}_\mathcal{D}(\pi\times \pi) = \mathrm{cd}(\pi)=k.$
\end{corollary}
\begin{proof} The space $\Bbb R^k$ is a free, contractible $\pi$-CW-complex where the action is given by $(a, x) \mapsto a+x$ for $a\in \Bbb Z^k$ and $x\in \Bbb R^k$. We may promote this $\pi$-action to a $G$-action on $\Bbb R^k$, by setting $((a,b), x)\mapsto a-b+x$ (here we use the assumption that $\pi$ is abelian). It is easily seen that $\Bbb R^k$ with this $G$-action becomes a model for 
$E_\mathcal{D}(G)$.
The inequality $\mathrm{cd}_\mathcal{D}(\pi\times \pi) \leq \mathrm{cd}(\pi)=k$ is now immediate and the inverse inequality $\mathrm{cd}(\pi)\leq \mathrm{cd}_\mathcal{D}(\pi\times \pi)$ is Proposition \ref{Shapiro}.
\end{proof}

\begin{corollary}\label{AtimesB}
Let $X$ be a finite aspherical complex with fundamental group $\pi=\pi_1(X,x_0)$, and let $K\leq G=\pi\times\pi$ be a subgroup such that $K\cap H = \{1\}$ for all $H\in \mathcal{D}$. Then $\mathrm{cd}(K)\leq\tc(X)$.
\end{corollary}
\begin{proof}
Under these assumptions any model for $E_\mathcal{D}(G)$ is free and contractible when viewed as a $K$-CW-complex, hence it is a model for $E(K)$. The same is true for $E(G)$. 
Letting $k:=\tc(X)$, we get a sequence of $G$-maps $E(G)\to L\to E_\mathcal{D}(G)$, where $L$ is a $G$-CW complex of dimension $k$. Restricting to $K$-actions we get, up to $K$-homotopy, a factorisation of the identity map $E(K)\to L\to E(K)$. This obviously implies that any cohomology class 
in $H^m(\pi, M)$ with $m>k$ vanishes, i.e. 
 $\mathrm{cd}(K)\leq k$, as stated.
\end{proof}

As a particular case of the above, let $K=A\times B$ where $A$ and $B$ are subgroups of $\pi$ such that $gAg^{-1}\cap B=\{1\}$ for all $g\in \pi$. We obtain that $\mathrm{cd}(A\times B)\leq \tc(\pi)$, which recovers the main result of \cite{GrantLuptonOprea}. 

\section{Lower bounds for $\tc(X)$ via Bredon cohomology}

In this section we shall give lower bounds for the topological complexity using Bredon cohomology. 

First we recall the basic constructions. 

\subsection{The family $\D$}\label{secd}
Let $\pi$ be a discrete group, we shall denote $G=\pi\times\pi$.  
As above, we denote by $\mathcal D$ the smallest family of subgroups $H\subset \pi\times \pi=G$ which contains the diagonal 
$\Delta\subset \pi\times\pi$, the trivial subgroup and which is closed under taking conjugations and finite intersections. 
It is easy to see that a nontrivial subgroup $H\subset \pi\times\pi$ belongs to $\D$ iff it is of the form
$$H_{b, S} \, =\, \{(a, bab^{-1}), \, a\in Z(S)\},$$
where $b\in \pi$ and $Z(S)$ denotes the centraliser of a finite set of elements $S\subset \pi$, i.e. 
$Z(S) =\{a\in \pi, sa=as\, \, \mbox{for any}\, \, s\in S\}.$

We denote by $\od$ {\it the orbit category} with objects transitive left $G$-actions having isotropy in $\D$ and with $G$-equivariant maps as morphisms, see \cite{Bre}. Objects of the category $\od$ have the form $G/H$ where $H\in \D$. 

\subsection{$\od$-modules and their principal components}\label{pcomp} {\it A (right) $\od$-module} $\um$ is a contravariant functor on the category of orbits $\od$ with values in the category of abelian groups. Such a module 
is determined by the abelian groups $\um(G/H)$ where $H\in \D$, and by a group homomorphism 
$$\um(G/H) \to \um(G/H')$$ associated with any $G$-equivariant map $G/H' \to G/H$ satisfying the usual compatibility conditions, expressing the fact that $\um$ is a functor. 

The abelian group $M= \um(G/1)$ is a left $\zpp$-module; an element $(g, h)\in \pi\times \pi$ acts on $\pi\times\pi$ by right translation and applying 
the functor $\um$ this defines an action on $M$. We shall call the $\zpp$-module $M$ {\it the principal component of } $\um$.

\begin{example} \label{free} {\rm Let $X$ be a left $G$-set. One defines an $\od$-module $\um_X$ by $$\um_X(?)=\Z[?, X]_G.$$ In other words, 
$\um_X(G/H)$ is the free abelian group generated by the set of $G$-equivariant maps $$[G/H, X]_G \, =\, X^H.$$ 
The homomorphism associated to a
morphism $f: G/H' \to G/H$ is the map $X^H\to X^{H'}$ given by $x\mapsto f(1, 1)\cdot x\in X^{H'}$ for $x\in X^H$. 

If the set $X$ is such that the isotropy subgroup of any point $x\in X$ belongs to the family $\D$ then the $\od$-module $\um_X$ is {\it free and projective}, 
see \cite{tomD}, chapter 1 or \cite{Luebook}, chapter 2. 

The principal component of the $\od$-module $\um_X$ is $M=\Z[X],$ the free abelian group generated by $X$. The left action of $G$ on $\Z[X]$ is induced by the left action of $G$ on $X$. 

Any equivariant map between $G$-sets $f:X\to Y$ induces naturally a homomorphism of the $\od$-modules $f_\ast: \um_X\to \um_Y$. 
 }\end{example}

Next we consider a few special cases of the previous example.

\begin{example} \label{free0}{\rm 
  Taking $X=\ast$, the one point orbit, we obtain the module $\um_X$ which will be denoted $\uz$. It associates 
$\Z$ to any orbit $\pi\times\pi/H$ with the identity homomorphism associated to any morphism of the orbit category $\od$.
Note that $\uz$ is not a free $\od$-module since $G=\pi\times\pi$ is not in $\D$. }\end{example}

\begin{example}\label{free1}{\rm  
In Example \ref{free} take $X=\pi$, the group $\pi$ viewed as a $G=\pi\times \pi$-set via the action $(x, y)\cdot g=xgy^{-1}$. The isotropy subgroup of an element $g\in \pi$ is 
$\{(x, g^{-1}xg), x\in \pi\}$ which belongs to the family $\D$ and hence the Bredon module $\um_\pi$ is free. Note that $\um_\pi$ associates the abelian group 
 $\Z[\pi^H]$ to any orbit $G/H$. 

If $H=H_{b, S}$ then $\pi^H$ coincides with $Z(Z(S))\cdot b^{-1}$. In general, $\pi^H$ is not a subgroup. 
}\end{example}

\begin{example}\label{frees}{\rm This is a generalisation of the previous example. For an integer $s\ge 1$, consider the $s$-th Cartesian power $\pi^s$ as a $G=\pi\times\pi$-set via the action 
$(x, y)\cdot (g_1, \dots, g_s)= (xg_1y^{-1}, \dots, xg_sy^{-1}).$ The isotropy subgroup of an element $(g_1, \dots, g_s)$ is the intersection of the isotropy 
subgroups of $g_i$ for $i=1, \dots, s$, hence it can be presented as  $H_{b, S}$ with $b=g_1$ and $S=\{g_1g_2^{-1}, g_1g_3^{-1}, \dots, g_1g_s^{-1}\}$. 
We obtain a free Bredon module $\um_{\pi^s}$, $s\ge 1$. Its principal component is the module $\Z[\pi^s]$. 

}\end{example}

\subsection{Bredon cohomology} Now we recall the construction of Bredon cohomology, see for example \cite{Mis}. 

Let $X$ be a $G$-CW-complex such that the isotropy subgroup of every point $x\in X$ belongs to the family $\D$. For every subgroup $H\in \D$
we may consider the cell complex $X^H$ of $H$-fixed points and its cellular chain complex $C_\ast(X^H)$. A $G$-map $\phi: G/K\to G/L$, where $K, L\in \D$, induces a cellular map $X^L\to X^K$ by mapping $x\in X^L$ to $gx\in X^K$ where $g$ is determined by the equation $\phi(K)=gL$ (thus $g^{-1}Kgx=x$ since $g^{-1}Kg\subset L$ and therefore $Kgx=gx$, i.e. $gx\in X^K$). Thus we see that the chain complexes $C_\ast(X^H)$, considered for all $H\in \D$, form a 
chain complex of right $\od$-modules which will be denoted ${\underline C}_\ast(X)$; here
${\underline C}_\ast(X)(G/H) = C_\ast(X^H).$
The principal component of the $\od$-chain complex ${\underline C}_\ast(X)$ is the chain complex $C_\ast(X)$ of left $\Z[G]$-modules. 

Note that the complex  ${\underline C}_\ast(X)$ is free as a complex of $\od$-modules although the complex $C_\ast(X)$ might not be free as a complex of $\Z[G]$-modules.

There is an obvious augmentation 
$\epsilon: {\underline C}_0(X)\to \uz$ which reduces to the usual augmentation $C_0(X^H)\to \Z$ on each subgroup $H\in \D$. 

If $\um$ is a right $\od$-module, we may consider the cochain complex of $\od$-morphisms $\Hom_\od ({\underline C}_\ast(X), \um)$.  
Its cohomology 
\begin{eqnarray}
H_\D^\ast(X; \um) \, = \, H^\ast(\Hom_\od ({\underline C}_\ast(X), \um))
\end{eqnarray}
is {\it the Bredon equivariant cohomology of $X$ with coefficients in $\um$. }

Let $M$ denote the principal component of $\um$.  By reducing to the principal components we obtain a homomorphism of cochain complexes
$$\Hom_\od ({\underline C}_\ast(X), \um)\to \Hom_{\Z[G]} ({C}_\ast(X), M)$$
and the associated homomorphism on cohomology groups
\begin{eqnarray}\label{red}
 H^i_\D(X; \um) \, \to \, H^i_G(X, M).
\end{eqnarray}

\subsection{} \label{sec34} If the action of $G$ on $X$ is free then obviously the homomorphism (\ref{red}) is an isomorphism and 
$$H^i_\D(X; \um) \, \cong  \, H^i(X/G, M),$$
where on the right we have the usual twisted cohomology. In particular we obtain
$$H^n_\D(E(\pi\times\pi), \um) = H^n(\pi\times\pi, M).$$

\subsection{} \label{sec35} Suppose now that $X=E(\pi)$, viewed as a left $G$-CW-complex, where $G=\pi\times\pi$, see \S \ref{sec24}. 
We know that $E(\pi)$ is a model for the classifying space $E_\D(G)$ (as we established in \S \ref{sec24}) and the classifying complex $E_\D(G)$ is unique up to $G$-homotopy.  Hence we may use the notation
$$H^\ast_\D(E(\pi), \um) = H^\ast_\D(\pi\times\pi, \um).$$
 We obtain
that the number $\cd_\D(\pi\times\pi)$ coincides with the minimal integer $n$ such that 
$H^i_\D(\pi\times\pi, \um) =0$
for all $i>n$ and for all $\od$-modules $\um$. 

\subsection{} Consider now the effect of the equivariant map $F: E(\pi\times\pi) \to E(\pi)$, see (\ref{F}). Note that any two equivariant maps 
$E(\pi\times\pi) \to E(\pi)$ are equivariantly homotopic. 
The induced map on Bredon cohomology 
$$F^\ast: H^i_\D(E(\pi), \um) \to H^i_\D(E(\pi\times\pi), \um)$$ 
in the notations introduced in \S \ref{sec34} and \S\ref{sec35} produces a homomorphism 
\begin{eqnarray}\label{Phi}
\Phi: H^i_\D(\pi\times\pi, \um) \, \to \, H^i(\pi\times\pi, M)
\end{eqnarray}
which connects the Bredon cohomology with the usual group cohomology. 

Now we may state a result which gives useful lower bounds for the topological complexity $\tc(X)$. 

\begin{theorem}\label{lower}
Let $X$ be a finite aspherical cell complex with fundamental group $\pi$. Suppose that
for some $\od$-module $\um$
there exists a Bredon cohomology class 
$$\underline \alpha\in H^n_\D(\pi\times\pi, \um)$$
such that the class 
$$\Phi(\underline \alpha)\not=0\in H^n(\pi\times\pi, M)$$
is nonzero. 
Then $\tc(X) \ge n$. Here $M$ denotes the principal component of $\um$. 
\end{theorem}

\begin{proof} Suppose that $\tc(X)<n$. Then by Theorem \ref{thm00} the map $F: E(\pi\times\pi) \to E(\pi)$ admits a factorisation
$$E(\pi\times\pi) \to L\to E(\pi)$$
where $L$ is a $G$-CW-complex of dimension less than $n$. Then the homomorphism 
$$\Phi: H^n_\D(\pi\times\pi, \um) \to H^n(\pi\times\pi, M)$$
factors as 
$$\Phi: H^n_\D(\pi\times\pi, \um) \to H^n_\D(L, \um) \to H^n(\pi\times\pi, M)$$
and the middle group vanishes since $\dim L <n$. This contradicts our assumption that $\Phi(\underline \alpha)\not=0$ for some $\underline \alpha\in H^n_\D(\pi\times\pi, \um)$. 
\end{proof}

\begin{remark}{\rm 
Theorem \ref{lower} can be compared to the classical result concerning the Lusternik - Schirelman category (see Eilenberg - Ganea \cite{EG} or Schwarz \cite{Sv66}) stating that for an aspherical space $X$ the existence of a nonzero cohomology class $H^n(X, M)$ (with some local coefficient system $M$) implies that $\cat(X)\ge n$. It is not true that $\tc(X)\ge n$ if $H^n(X\times X, M)\not=0$ for $X$ aspherical. For example, in the case of the circle $X=S^1$ we know that 
$\tc(X)=1$ while $H^2(X\times X,\Z)\not=0$. Theorem \ref{lower} imposes a condition on the nontrivial cohomology class in the usual twisted cohomology 
to be extendable to a class in Bredon cohomology. We will investigate this property further in  \S \ref{essential}. 
}
\end{remark}

\section{The canonical class in Bredon cohomology and its universality}

In this section we define a special Bredon cohomology class which will play an important role in this paper.

\subsection{The canonical class} \label{secan} Consider an $\od$-module $M_X(?)=\Z[?, X]_G$ where $X$ is a $G$-set, see
Example \ref{free}. Recall that $G$ denotes the group 
$\pi\times\pi$. The unique map $X\to \ast$ is $G$-invariant and induces a homomorphism of Bredon modules
$\epsilon: \um_X \to \um_\ast =\uz$, called {\it the augmentation}. We denote by $\ui_X$ the kernel of $\epsilon$. Clearly, $\ui_X$ is a Bredon module whose value 
on an orbit $G/H$ is $$\ui_X(G/H) = \ker[\epsilon: \Z[X^H] \to \Z].$$ 

As a special case of the previous construction we obtain the Bredon module $\ui_\pi$ (where $X=\pi$, as in Example \ref{free1}). 
Here 
\begin{eqnarray}\label{ideal}
\ui_\pi(G/H) = \ker[\epsilon: \Z[\pi^H] \to \Z] \, \equiv\, I (\pi^H).
\end{eqnarray} 
We shall shorten the notation $\ui_\pi$ to $\underline I$. 
The principal component of $\ui$ is the augmentation ideal $I=\ker[\epsilon: \zp\to \Z]$. 

One obtains a short exact sequence of Bredon modules
\begin{eqnarray}\label{sec1}
0\to \ui \to \um_\pi\stackrel{\epsilon}\to \uz \to 0.
\end{eqnarray}
The latter defines a Bredon cohomology class $$\uu\in \Ext_{\od}^1(\uz, \ui) \, \equiv \, H^1_\D(\pi\times\pi, \ui).$$ 
We shall call $\uu$ {\it the canonical class in Bredon cohomology}. It is a refinement of {\it the ordinary canonical class} $$\vv\in H^1(\pi\times \pi, I)$$ which  was defined in \cite{CosFar}. In \cite{FM}, \S 3 it is shown that $\vv$ coincides with the class represented by the 
principal components  of the sequence (\ref{sec1}), i.e. by the exact sequence of left $\zpp$-modules
\begin{eqnarray}\label{sec2}
0\to I \to \zp \to \Z \to 0.
\end{eqnarray}
Hence,
the principal component of the class $\uu$ (i.e. the image of $\uu$ under the homomorphism (\ref{Phi})), coincides with $\vv$. 

The canonical class $\vv$ is closely related to {\it the Berstein - Schwarz class} $${\mathfrak b}\in H^1(\pi, I)$$ which is represented by the exact sequence (\ref{sec2}) viewed as a sequence of left $\zp$-modules. 

%
%
%

\subsection{The classes $\uu^n$} Next we define classes $$\uu^n\in H^n_\D(\pi\times\pi, \ui^n), \quad n=1, 2, \dots.$$ 
 In this paper we shall treat these classes formally and call them {\it the powers of the canonical class} $\uu$ without trying to justify this name. However we shall show that the principal component of the class $\uu^n$ is the $n$-fold cup product 
$\vv\cup\vv\cup\dots\cup \vv=\vv^n$ of the canonical class $\vv\in H^1(\pi\times\pi, I)$. 

The Bredon module $\ui^n$ is defined by 
$$\ui^n(G/H) = I(\pi^H)\otimes_\Z I(\pi^H)\otimes_\Z \dots\otimes_\Z I(\pi^H), \quad H\in \D.$$
We shall define the class $\uu^n$ by describing an explicit exact sequence of $\od$-modules
\begin{eqnarray}\label{derham}
0\to \ui^n \to \underline C_{n-1}\stackrel{d}\to \underline C_{n-2}\stackrel{d}\to \dots\stackrel{d} \to \underline C_0\to \uz \to 0
\end{eqnarray}
in which the intermediate $\od$-modules $\underline C_0, \underline C_1, \dots, \underline C_{n-1}$ are projective. 
If $$\uP_\ast: \quad \cdots \uP_2\to \uP_1\to \uP_0\to \uz\to 0$$ is an $\od$-projective resolution of $\uz$, we obtain a commutative diagram (unique up to chain homotopy)
$$
\begin{array}{cclccccclcc}
\uP_{n+1}&\to &\uP_n&\to \uP_{n-1}&\to &\cdots& \uP_0&\to&\uz&\to& 0\\ \\
\downarrow && \downarrow f &\downarrow &&&   \downarrow  &&\downarrow = &\\ \\
0&\to &\ui^n&\to C_{n-1}&\to &\cdots& C_0&\to&\uz&\to& 0.
\end{array}
$$
The $\od$-homomorphism $f$ is a cocycle, and its cohomology class 
$$\{f\}\in H^n(\Hom_\od(\uP_\ast, \ui^n))=H^n_\D(\pi\times\pi, \ui^n)$$ is independent of the choice of the chain map represented by the diagram above. We define the {\it $n$-th power of the canonical class} $\uu^n$ as the cohomology class $\{f\}$.

The principal components of the exact sequence (\ref{derham}) define an exact sequence of left $\zpp=\Z[G]$-modules
$$
0\to \ui^n(G/1)=I^n  \to \underline C_{n-1}(G/1)\stackrel{d}\to \underline C_{n-2}(G/1)\stackrel{d}\to \dots\stackrel{d} \to \underline C_0(G/1)\to \Z \to 0. 
$$
This sequence determines a class in $$\Ext_{\zpp}^n(\Z, I^n)=H^n(\pi\times\pi, I^n)$$ which is {\it the principal component of the class $\uu^n$}. 
We shall identify the principal component of $\uu^n$ with $\vv^n$, see Theorem \ref{thm3}.

\subsection{Construction of the complex (\ref{derham})}\label{construction} Here we shall generalise a construction of Dranishnikov and Rudyak \cite{DranRud}; see also \cite{FM}. 

We shall use the operation $\otimes_\Z$ of tensor product of $\od$-modules which is defined as follows. For two right $\od$-modules $\um$ and $\un$ we define $\um\otimes_\Z\un$ by the formula
$$
\left(\um\otimes_\Z\un\right)(G/H) = \um(G/H)\otimes_\Z\un(G/H), \quad H\in \D,
$$
with the obvious action on morphisms. 

The following obvious remark will be used in the sequel. Suppose that 
$$0\to \um_1\to \um_2\to \um_3\to 0$$ is an exact sequence of right $\od$-modules and let $\un$ be a right $\od$-module such that 
 for any $H\in \D$ the module
$\un(G/H)$ is free as an abelian group. Then the sequence 
$$0\to \un\otimes_\Z\um_1 \to \un\otimes_\Z\um_2\to \un\otimes_\Z \um_3\to 0$$ 
is also exact. 

Let $X$ and $Y$ be left $G$-sets, where $G=\pi\times\pi$. Consider the $\od$-modules $\um_X$ and $\um_Y$, see Example \ref{free}. 
Note that the tensor product 
$\um_X\otimes_\Z \um_Y$ can be naturally identified with $\um_{X\times Y}$. 
We know that the modules $\um_X$, $\um_Y$ and  $\um_{X\times Y}$ are free iff the 
isotropy subgroups of all elements of $X$ and $Y$ belong to $\D$. 

Tensoring the short exact sequence 
\begin{eqnarray}\label{mxy}
0\to \ui_Y \to \um_Y \stackrel{\epsilon}\to \uz\to 0
\end{eqnarray}
 with $\um_X$ we obtain an exact sequence
\begin{eqnarray}\label{sec5}0\to \um_X \otimes_\Z \ui_Y\to \um_{X\times Y} \to \um_X\to 0\end{eqnarray}
in which $\um_X$ and $\um_{X\times Y}$ are free and hence the sequence (\ref{sec5}) splits. 
We conclude: 
{\it If the isotropy subgroups of all elements of  $X\sqcup Y$ belong to $\D$, then the $\od$-module
$$\um_X\otimes_\Z \ui_Y$$ is projective}. 
Taking in the above statement $X=\pi$ and $Y=\pi^r$, where $\pi^r$ is equipped with the left $\pi\times\pi$ action 
$(x, y)\cdot(a_1, \dots, a_r)= (xa_1y^{-1}, \cdots, xa_ry^{-1}),$ 
we obtain that 
{\it the $\od$-module $$\um_\pi \otimes_\Z \ui^{r}$$ is projective for any $r\ge 0$. } 
Here $\ui^r$ denotes the $r$-fold tensor product $\ui\otimes_\Z\ui \otimes_\Z\dots \otimes_\Z \ui$.

Starting from the short exact sequence (\ref{sec1}) and tensoring with $\ui$ we iteratively obtain short exact sequences of $\od$-modules
\begin{eqnarray}\label{20}
0\to \, \ui^r\, \stackrel{i\otimes 1}\to \, \um_\pi\otimes_\Z \ui^{r-1}\, \stackrel{\epsilon\otimes 1}\to \, \ui^{r-1} \to 0,\quad r=1, 2, \dots.
\end{eqnarray}
Splicing them for $r=1, 2, \dots, n$ we obtain the long exact sequence of $\od$-modules
\begin{eqnarray}\label{mpi}
0\to \ui^n\to \um_\pi\otimes_\Z \ui^{n-1} \to  \um_\pi\otimes_\Z \ui^{n-2} \to \dots\to \um_\pi\otimes_\Z \ui\to \um_\pi\to \uz\to 0.
\end{eqnarray}
This is a version of the complex (\ref{derham}). 
Naturally, there exist many other chain complexes representing the same cohomology class $\uu^n$.

\subsection{Universality of the canonical class} In this subsection we prove the following statement which is a generalisation of the well-known result of A.S. Schwarz
(see \cite{Sv66}, Proposition 34). 

\begin{theorem}\label{univ} For any $\od$-module $ \um$ and for any cohomology class $$\ualpha\in H^n_\D(\pi\times\pi, \um)$$ there exists
an $\od$-morphism $\phi: \ui^n\to  \um$ such that $\phi_\ast(\uu^n) = \ualpha$. 
\end{theorem}
\begin{proof}
One may construct a projective $\od$-resolution of $\uz$ extending (\ref{derham})
$$\dots \stackrel{d}\to \underline C_n\stackrel{d}\to \underline C_{n-1}\stackrel{d}\to \underline C_{n-2}\stackrel{d}\to \dots\stackrel{d}\to \underline C_0\to \uz\to 0.$$
The class $\ualpha$ can be viewed as a cohomology class of the cochain complex $\Hom_\od (\underline C_\ast,  \um)$.  
Let $f: \underline C_n\to \um$ be a cocycle representing $\alpha$. In the diagram
$$
\begin{array}{ccccccc}
\underline C_{n+1}& \stackrel{d}\to & \underline C_n& \stackrel{d}\to & \ui^n &\stackrel{d}\to & 0\\ \\
&&\downarrow f&\swarrow \phi& &&\\ \\
&&\um&&&&
\end{array}
$$
the row is exact and the existence of a $\od$-homomorphism $\phi: \ui^n\to \um$ follows from the assumption that $f$ is a cocycle. 
We claim that $\phi_\ast(\uu^n) = \ualpha$. Indeed, the class $\uu^n$ is represented by a similar diagram 
$$
\begin{array}{ccccccc}
\underline C_{n+1}& \stackrel{d}\to & \underline C_n& \stackrel{d}\to & \ui^n &\stackrel{d}\to & 0\\ \\
&&\downarrow g&\swarrow {\rm {id}}& &&\\ \\
&&\ui^n&&&&
\end{array}
$$
implying that $\phi\circ g =f$. Hence we see that the cocycle representing the class $\ualpha$ is obtained from the cocycle representing $\uu^n$ by composing with $\phi$. 
\end{proof}

Theorem \ref{univ} obviously implies:

\begin{corollary}
One has 
$$\cd_\D(\pi\times\pi)= {\rm {height}}(\uu),$$
where the integer ${\rm {height}}(\uu)$ is defined as the largest $n$ such that the class $\uu^n\in H^n_\D(\pi\times\pi, \ui^n)$ is nonzero. 
\end{corollary}


\begin{theorem}\label{thm3} For any integer $n\ge 1$
the image of the class $$\uu^n\in H^n_\D(\pi\times\pi, \ui^n)$$ 
 under the homomorphism (\ref{Phi}) coincides with the $n$-fold cup-power
$$\Phi(\uu^n) \,  =\, \vv^n \, =\, \vv\cup \vv\cup \dots\cup \vv\, \in\,  H^n(\pi\times\pi, I^n)$$
of the canonical class $\vv\in H^1(\pi\times\pi, I)$. 
\end{theorem}
\begin{proof} The principal components of the complex (\ref{mpi}) is exactly the chain complex (12) from \cite{FM} and our statement is identical to Lemma 3.1 from \cite{FM}. 
\end{proof}

\section{Principal $\od$-modules}

In this section $G$ denotes the group $\pi\times\pi$ and $\D$ is the family of subgroups of $G$ defined in \S \ref{secd}

\subsection{} Let $\um$ be an $\od$-module. {\it The principal component} of $\um$ is defined as 
$\um(G/1) = A$ which (as we noted in \S \ref{pcomp}) has the structure of a left $\Z[G]$-module. Note that for any orbit $G/H$
we have an $\od$-morphism $f_H: G\to G/H$ given by $g\mapsto gH$ which induces a homomorphism 
$$\um(f_H): \um(G/H) \to \um(G/1)=A.$$
For $a\in H$ we have $f_H =f_H\circ r_a$ 
%
%
where $r_a: G\to G$ is the right multiplication by $a$, i.e. $r_{a}(g)=ga$. Applying the functor $\um$ we see that the homomorphism 
$\phi_H\equiv \um(f_H)$ takes values in $A^H$, i.e. 
\begin{eqnarray}\label{ah}
\phi_H: \um(G/H) \to A^H. 
\end{eqnarray}
\begin{definition} \label{defprin} We shall say that an $\od$-module $\um$ is principal if for any subgroup $H\in \D$ 
the homomorphism 
\begin{eqnarray}\label{isom}\phi_H=\um(f_H)\, :\,  \um(G/H)\to A^H\end{eqnarray} 
is an isomorphism. 
\end{definition}
Let $\um$ be a principal $\od$-module. Let $H, K\in \D$ and let $a\in G$ be such that $a^{-1}Ha\subset K$. Then we have an $\od$-morphism
$f_a: G/H\to G/K$ where $f_a(gH) = gaK$ for any $g\in G$. 
We obtain the commutative diagram
$$
\begin{array}{ccc}
G& \stackrel{r_a}\to & G\\
f_H \downarrow & & \downarrow f_K\\
G/H & \stackrel{f_a}\to & G/K.
\end{array}
$$
of orbits and applying the functor $\um$ we obtain the commutative diagram 
$$
\begin{array}{ccc}
A^K& \stackrel{r_a^\ast}\to & A^H\\ 
\phi_K \uparrow \simeq & & \simeq \uparrow \phi_H\\ 
\um(G/K)& \stackrel{f_a^\ast}\to & \um(G/H).
\end{array}
$$
where $r_a^\ast$ is multiplication by $a$. 
Thus we see that the structure of a principal $\od$-module $\um$ is fully determined by the left $\Z[G]$-module 
$A$ (the principal component of $\um$). 
Viewing $A$ as a left $G$-set we may write
$$\um(G/H) = [G/H, A] = A^H.$$

Principal modules appear in the book of G. Bredon \cite{Bre} as Example (2), page I-10. 

\subsection{} As an example consider the $\od$-module $\um_X(?)=\Z[?, X]_G$ (see Example \ref{free}) where $X$ is a left $G$-set. In this case the principal component is 
$\Z[X]$ viewed as a left $\Z[G]$-module. For an orbit $G/H$ with $H\in \D$ we have $\um_X(G/H) =\Z[X^H]$ and the map 
$f_H: G/1 \to G/H$ induces a homomorphism 
\begin{eqnarray}\label{xh}
\Z[X^H]\to (\Z[X])^H
\end{eqnarray}
which in general is an inclusion. 

\begin{lemma}\label{prin1}
The homomorphism (\ref{xh}) is an isomorphism if and only if for any $H\in \D$, the set $X$ viewed as an $H$-set, has the following property: any 
orbit of $H$ contained in $X$ is either infinite or a single point. 
\end{lemma}
\begin{proof} Suppose that $X$ satisfies the condition of the Lemma. For $H\in \D$ we may split $X$ into a disjoint union of $H$-orbits $X=\sqcup_j X_j$
where each $X_j$ is either a single point or infinite. Then $\Z[X] =\oplus_j \Z[X_j]$ and $\Z[X]^H =\oplus_j \Z[X_j]^H$ with $\Z[X_j]^H = \Z[X_j]$ if $X_j$ is a single point and $\Z[X_j]^H = 0$ if $X_j$ is infinite. On the other hand the set $X^H$ is the union of the sets $X_j$ which are single points. Hence (\ref{ah}) is an isomorphism. 

The inverse statement follows similarly. Namely, suppose that $X_j\subset X$ is a finite $G$-orbit which is not a single point. Then the element
$$\sum_{x\in X_j} x\, \, \in\,  \Z[X]$$
is invariant with respect to $H$, i.e. it lies in $\left(\Z[X]\right)^H$ but not in $\Z[X^H]$. 
\end{proof}

We want to restate Lemma \ref{prin1} in terms of the isotropy subgroups of points of $X$. For a point $x\in X$ denote by $I(x)\subset \pi\times\pi$ its isotropy subgroup. For a subgroup $H\subset G=\pi\times\pi$ one has $x\in X^H$ iff $H\subset I(x)$. The orbit of $x$ with respect to $H$ is finite iff 
$H$ contains $I(x)\cap H$ as a finite index subgroup. Thus we obtain the following Corollary:

\begin{corollary}\label{cor34a}
The $\od$-module $\um_X$ is principal if and only if for any $x\in X$ and any subgroup $H\in \D$ the index $[H:H\cap I(x)]$ is either $1$ or $\infty$. 
\end{corollary}

For free $\od$-modules $\um_X$ the set $X$ has all isotropy subgroups in $\D$. This leads to the following Corollary:

\begin{corollary} \label{cor34} Suppose that for any two subgroups $H, H'\in \D$ the index $[H:H\cap H']$ is either 1 or $\infty$. Then any free $\od$-module is principal. 
\end{corollary}

Note that the property of the family of subgroups $\D$ described in Corollary \ref{cor34} is in fact a property of the group $\pi$ since the family $\D$ depends on the group
$\pi$ alone. 

\begin{definition}\label{defprincipal}
We shall say that a group $\pi$ is principal if any of the following equivalent conditions is satisfied:
\begin{enumerate}
\item[{(a)}] Any free $\od$-module is principal, 
\item[{(b)}] For any two subgroups $H, H'\in \D$, the index $[H:H\cap H']$ is either 1 or infinity, 
\item[{(c)}] For any two finite subsets $S, S'\subset \pi$ the group $Z(S)/Z(S\cup S')$ is either infinite or trivial.
\end{enumerate}
\end{definition}

Recall that the symbol $Z(S)$ denotes the centraliser of $S$, i.e. the set of all elements $g\in \pi$ which commute with every element of $S$. 
The equivalence between (a) and (b) follows from Corollaries \ref{cor34a} and \ref{cor34}. The equivalence $(b)\sim (c)$ follows from the structure of the groups  $H\in \D$.

\begin{example}{\rm 
Let $\pi=\Z^n$. Then the class $\D$ contains only two subgroups, the trivial subgroup and the diagonal $\Delta$. The condition of Corollary \ref{cor34} is clearly satisfied, i.e. $\Z^n$ is a principal group. }
\end{example}
Other examples of principal groups will be described in \S \ref{sec:examples}.


\begin{lemma} \label{lm622} Let 
$0\to \um_1\stackrel{\alpha}\to \um_2\stackrel{\beta}\to \um_3$
be an exact sequence of $\od$-modules such that the modules $\um_2$ and $\um_3$ are principal. 
Then the module $\um_1$ is also principal. 
\end{lemma}
\begin{proof}  Denote $G=\pi\times\pi$ for short. 
For any $H\in \D$ we have the following commutative diagram 
$$
\begin{array}{ccccccc}
0&\to & M_1(G/H)& \stackrel{\alpha}\to & M_2(G/H)& \stackrel{\beta}\to & M_3(G/H)\\ \\
&& \downarrow \phi^1_H && \downarrow \phi^2_H && \downarrow \phi^3_H \\ \\
0 & \to & A_1^H & \stackrel{\alpha}\to & A_2^H & \stackrel{\beta}\to & A_3^H
\end{array}
$$
The rows are exact and $\phi^2_H$ and $\phi^3_H$
are isomorphisms. 
By the 5-lemma we obtain that $\phi_H^1$ is also an isomorphism. Hence $\um_1$ is principal. 
\end{proof}

Lemma \ref{lm622} can also be stated as saying that the kernel of a $\od$-morphism of principal Bredon modules is principal. 

\begin{corollary}\label{cor48}
Assume that the group $\pi$ is principal. Then the $\od$-module $\ui^n$ is principal for any $n\ge 1$. 
\end{corollary} 
\begin{proof} First let us make the following general remark. Let $X$ and $Y$ be left $\pi\times\pi$-sets with all isotropy subgroups in $\D$. 
Then the $\od$-module $\um_X\otimes_\Z \ui_Y$ is principal as follows 
by applying 
Lemma \ref{lm622} to the exact sequence (\ref{sec5}) and noting that the free modules $\um_{X\times Y}$ and $\um_X$ are principal. 

The statement of 
Corollary \ref{cor48} now follows by inductively applying the above remark to the exact sequence (\ref{20}). \end{proof}

Morphisms between principal modules are determined by their effect on the principal components:

\begin{lemma} \label{thm4} Let $\um_1$ and $\um_2$ be principal $\od$-modules. Let $A_1$ and $A_2$ be their principal components. 
Then the map 
\begin{eqnarray}\label{mor}
\Hom_\od(\um_1, \um_2) \to \Hom_\zpp(A_1, A_2),
\end{eqnarray}
associating with any morphism its effect on the principal components, is an isomorphism. 
\end{lemma} 
\begin{proof} Let $f:\um_1\to \um_2$ be an $\od$-morphism. The map (\ref{mor}) associates with $f$ the $\zpp$-homomorphism 
$f_1: \um_1(G/1)=A_1\to \um_2(G/1)=A_2$. 
We have the following commutative diagram
$$
\begin{array}{ccc}
\um_1(G/H)& \stackrel{\phi^1_H}\to & A_1^H\\ \\
\downarrow f_H&&\downarrow f_1^H\\ \\
\um_2(G/H)& \stackrel{\phi^2_H}\to & A_2^H\end{array}
$$
in which $\phi_H^1$ and $\phi_H^2$ are isomorphisms. 
Thus, we see that the homomorphism $f_H$ is uniquely determined by the 
restriction $f_1^H$ of $f_1$ onto $A_1^H$. 
\end{proof}
\begin{corollary}\label{cor49}
Let $\uc_\ast$ be a chain complex of principal $\od$-modules and let $\um$ be a principal $\od$-module. 
Then the canonical map 
$$\Hom_\od(\uc_\ast, \um) \to \Hom_\zpp(C_\ast, M)$$
is an isomorphism of chain complexes. 
Here $C_\ast = \uc_\ast(G/1)$ is the principal component of $\uc_\ast$ and $M=\um(G/1)$ is the principal component of $\um$. 
\end{corollary}
\begin{proof}
This follows from Lemma \ref{thm4}. 
\end{proof}

\begin{corollary}\label{isomorphism}
Suppose that the group $\pi$ is principal. 
Let $C_\ast$ be the chain complex of left $\zpp$-modules consisting of principal components of a projective
$\od$-resolution of $\uz$. 
Then the natural map
$$ H^n_\D(\pi\times\pi, \ui^n)\to H^n(\Hom_\zpp(C_\ast, I^n)),$$ is an isomorphism. 
\end{corollary}
\begin{proof}
We apply Corollary \ref{cor49} to a $\od$-free resolution of $\uz$ noting that under our assumptions the Bredon module $\ui^n$ is principal (by Corollary \ref{cor48}). 
\end{proof}
\subsection{} Note that the complex $C_\ast$ which appears in Corollary \ref{isomorphism} is a resolution of $\Z$ over the ring $\zpp$ but it is neither free nor projective. Any projective resolution $P_\ast$ admits a chain map $P_\ast\to C_\ast$ and for any left $\zpp$-module $A$ we have a chain map
$\Hom_\zpp(C_\ast, A) \to \Hom_\zpp(P_\ast, A)$ (which is unique up to homotopy) inducing a well-defined homomorphism 
$$H^\ast(\Hom_\zpp(C_\ast, A))\to H^\ast(\Hom_\zpp(P_\ast, A))= H^\ast(\pi\times\pi, A).$$

\section{Essential cohomology classes}\label{essential}

The following notion was introduced and studied in \cite{FM}.

\begin{definition} Let $A$ be a left $\zpp$-module. 
A cohomology class $\beta\in H^n(\pi\times \pi, A)$ is said to be {\it essential} if there exists a homomorphism of $\Z[\pi\times \pi]$-modules $\mu: I^n \to A$ such that $$\mu_\ast(\vv^n)=\beta.$$  Here $\vv^n\in H^n(\pi\times\pi, I^n)$ denotes the $n$-th power of the canonical class $\vv$. 
\end{definition} 

In \cite{FM} the authors constructed a spectral sequence giving a full set of obstructions for a cohomology class to be essential. 
The first such obstruction is the requirement for the class $\beta\in H^n(\pi\times \pi, A)$ to be {\it a zero-divisor}, i.e. 
\begin{eqnarray}\label{divisor}\beta|_\pi=0\in H^n(\pi, A|_\pi)\end{eqnarray}
where $\pi\subset \pi\times\pi$ denotes the diagonal subgroup; see \cite{FM}, \S 5. The condition (\ref{divisor}) is obvious since the canonical class
$\vv$ and all its powers $\vv^n$ are zero-divisors. 

Here we characterise the essential cohomology classes as principal components of Bredon cohomology classes. 

\begin{theorem} \label{thm8}  Let $A$ be a left $\zpp$-module which is the principal component of an $\od$-module $\um$. 
Consider the homomorphism 
\begin{eqnarray}
\Phi: H^n_\D(\pi\times\pi, \um)\to H^n(\pi\times\pi, A)
\end{eqnarray} 
which associates to a Bredon cohomology class its principal component, see   (\ref{Phi}). 

(1) Any class $\beta\in H^n(\pi\times \pi, A)$ in the image of $\Phi$ is essential. 
 
 (2) If the group $\pi$ is principal then the set of essential cohomology classes coincides with the image on $\Phi$. 
\end{theorem}

\begin{proof} Suppose that $\beta=\Phi(\alpha)$ where $\um\in H^n(\pi\times\pi, \um)$. By the Universality Theorem \ref{univ}, there exists a $\od$-homomorphism $\mu: \ui^n\to \um$ such that $\alpha=\mu_\ast(\uu^n)$. On the principal components we obtain a $\zpp$-homomorphism 
$\mu: I^n \to A$
such that $\mu_\ast(\vv^n)=\beta.$ 
Thus $\beta$ is essential. Here we used Theorem \ref{thm3} stating that the principal component of $\uu^n$ is $\vv^n$. This proves statement (1).

Suppose now that a cohomology class $\beta\in H^n(\pi\times\pi, A)$ is essential, i.e. $\beta=\mu_\ast(\vv^n)$ where $\mu: I^n\to A$ is a $\zpp$-homomorphism. 
Let $\um$ denote the $\od$-module 
$$\um(G/H) \, = \, A^H \, =\, [G/H, \, A]_G.$$
whose principal component is $A$.
Here we view $A$ as a left $G$-set and the brackets $[\, \, , \, \, ]_G$ denote the set of $G$-maps. Since we assume that $\pi$ is principal we know that 
$\od$-module $\ui^n$ is principal (see Corollary \ref{cor48}). Applying Lemma \ref{thm4} we obtain a $\od$-morphism 
$\hat\mu: \ui^n\to \um$ having $\mu$ as its principal component. 
This produces a Bredon cohomology class
$$\alpha=\hat\mu_\ast(\uu^n)\in H^n_\D(\pi\times\pi, \um),$$ and using Theorem \ref{thm3} we have $\Phi(\alpha)=\mu_\ast(\vv^n) =\beta$. 

This completes the proof. 
\end{proof}



\section{Examples of principal groups}\label{sec:examples}

In this section we show that all torsion free hyperbolic groups as well as all torsion free nilpotent groups are principal. Also, we give an example of a non-principal group. 

\begin{definition}\label{def:propN}
We say that a group $\pi$ satisfies Property $N$ if, for any $a\in \pi$ and any finite set $S \subset \pi$, 
the inclusion $a^n \in Z(S)$,  where $n\ge 1$, implies that $a\in Z(S).$
\end{definition}

%
%
%
%
%
%

\begin{proposition}\label{prop:propN}
Any group $\pi$ satisfying Property $N$ is principal.
\end{proposition}

\begin{proof}
We shall use the property (c) from Definition \ref{defprincipal}. To show that the group $\pi$ is principal we need to show that for any two finite subsets 
$S, S'\subset \pi$ the group $Z(S)/Z(S\cup S')$ is either trivial or infinite. This will follow once we show that this group is torsion free. An element of order $n$ in $Z(S)/Z(S\cup S')$ is represented by an element $a\in Z(S)$ such that $a^n\in Z(S\cup S')$. But then Property N implies $a\in Z(S\cup S')$ i.e. 
$a$ represents the trivial class in $Z(S)/Z(S\cup S')$. 

\end{proof}

\begin{proposition}\label{prop:propNnilp}
If $\pi$ is a finitely generated torsion free nilpotent group, then $\pi$ satisfies Property $N$ and therefore $\pi$ is principal.
\end{proposition}

\begin{proof}
If $\pi$ is abelian, then $Z(S)=\pi$, so Property $N$ holds tautologically.
Suppose inductively that any finitely generated torsion free nilpotent group of class $<r$ satisfies Property $N$.
Take $\pi$ of class $r$ and let $a^n \in Z(S)$ for some $S$. Denote the quotient of $\pi$ by its centre by $\bar\pi = \pi/Z(\pi)$ and
note that: (1) The class of $\bar\pi$ is $<r$, so $\bar\pi$ satisfies Property $N$ and (2) $Z(S)$ maps
into $Z(\bar S)$ under the quotient map $\pi \to \bar\pi$. Then we see that $\bar a^n \in Z(\bar S)$
and, by Property $N$, we have $\bar a \in Z(\bar S)$. Let $g \in S$ so that $\bar g \in \bar S$. Then we see that
$[\bar a, \bar g]=1$ and this implies that $[a,g]\in Z(\pi)$. Let's now employ a basic relation among higher commutators
(which holds for any group [cf. Hall, 10.2.12]), 
$$[xy, z] = \big[ x, [y, z] \big]\, [y,z] \, [x,z].$$
Recall  that we have $[a^n, g] = 1$.  Expanding $[a^n, g]$ using the relation above gives
$$[a^n, g] = \big[ a^{n-1}, [a,g]  \big]\, [a,g] \, [a^{n-1},g] = [a,g] \, [a^{n-1},g],$$
where the last equality follows because $[a,g] \in Z(\pi)$.  Repeating this step eventually leads to $1 = [a^n, g] =  [a,g]^{n}$.  
Since $\pi$ is torsion free, we have $[a, g] = 1$ and $a\in Z(S)$. This completes the inductive step.
\end{proof}

\begin{lemma}\label{sur}
Let $\pi$ be a torsion free group such that the centraliser $Z(g)$ of any element $g\in \pi-\{1\}$  is cyclic. Then any two centralisers 
$Z(g_1)$, $Z(g_2)$, where $g_1, g_2\in \pi-\{1\}$, either coincide $Z(g_1)=Z(g_2)$ or their intersection is trivial, $Z(g_1)\cap Z(g_2)=\{1\}$.
\end{lemma}
\begin{proof} Let $a_i\in Z(g_i)$ be a generator, $i=1, 2$.
Assume that the intersection $Z(g_1)\cap Z(g_2)$ is not trivial. 
Then this intersection is an infinite cyclic group. Let $x\in Z(g_1)\cap Z(g_2)$ denote a generator of the intersection. Then 
\begin{eqnarray}\label{xg}
x=a_1^{n_1}=a_2^{n_2}
\end{eqnarray} for some $n_1, n_2\not=0$. 
Consider the centraliser $Z(x)\subset \pi$. It is an infinite cyclic group (by our assumption) which contains $a_1$ and $a_2$ (because of (\ref{xg})) implying that the elements $a_1$ and $a_2$ commute. 
Hence $Z(g_1)=Z(g_2).$

\end{proof}

\begin{lemma}\label{ex2}
Assume that a group $\pi$ is torsion free and the centraliser of any nontrivial element $g\in \pi$ is cyclic.  
Then $\pi$ satisfies property N and hence it is principal.
\end{lemma}
\begin{proof} Let $S\subset \pi$ be a finite subset. By Lemma \ref{sur}, if $Z(S)$ is nontrivial then $Z(S)=Z(g)$ for some $g\in \pi-\{1\}$. 
If $a^n\in Z(S)$ then $a^n\in Z(a)\cap Z(g)$. We know that the centralisers $Z(a)$ and  $Z(g)$ either coincide or have trivial intersection. 
If $Z(a)=Z(g)$ then $a\in Z(g)=Z(S)$. In the case $Z(a)\cap Z(g)=1$ we obtain $a^n=1$ and hence $a=1$ since $\pi$ is torsion free. 
%
%
%
%
%
\end{proof}

\begin{corollary}
Any torsion free hyperbolic group is principal.
\end{corollary}
\begin{proof}
This follows from Lemma \ref{ex2} since in a torsion free hyperbolic group the centraliser of any non-unit element is cyclic. 
\end{proof}

As example of a group that is not principal we have the following: 

\begin{example} \label{Klein} {\rm Consider the fundamental group $K$ of the Klein bottle, 
$$K=\langle c, d; c^2 = d^2\rangle.$$
Denote $z=c^2=d^2;$ this element generates the centre $Z\subset K$. 
Denote $x=cd$, $y=dc$. 
Any element of $K$ can be uniquely written in one of the four forms
$$x^kz^l, \, y^kz^l,\,  x^kz^lc,\,  y^kz^ld, \quad k, l \in \Z.$$
Relations:
$$xy=yx=z^2$$
$$cx=yc$$
$$dx=yd$$
$$cy=xc$$
$$dy=xd$$
We see that the centraliser of $x$ is the subgroup generated by $x, y$ and $z$. Note that $Z(x)\subset K$ is normal. 
Besides, $c\notin Z(x)$ while $c^2=z\in Z(x)$. This shows that $K$ does not have 
property $N$. 

Besides, the centraliser of $xy=z^2$ is the whole group $K$. In this case the group $K/K\cap Z(x)=K/Z(x)$ is $\Z_2$. 
Consider the following two subgroups $H, H'\subset K\times K$. Let $H=\Delta\subset K\times K$ be the diagonal and let 
$H'$ be $H'= \{(a, xax^{-1}); a\in K\}$. Then $H\cap H'= \{(a, a); a\in Z(x)\}$ and hence $H/H\cap H'\simeq K/Z(x)\simeq \Z_2$. 
We conclude that the fundamental group of the Klein bottle $K$ is not principal. 
}
\end{example}

%
%
%

%
%
%
%

\end{document}